\documentclass[a4paper]{amsart}
\usepackage{mymacros}
\usepackage{hyperref}
\usepackage[disable]{todonotes}
\usepackage{tensor}
\usepackage{verbatim}

\hypersetup{final}

\makeatletter
\newtheorem*{rep@theorem}{\rep@title}
\newcommand{\newreptheorem}[2]{
  \newenvironment{rep#1}[1]{
    \def\rep@title{#2 \ref{##1}}
    \begin{rep@theorem}}
    {\end{rep@theorem}}}
\makeatother

\newreptheorem{ex}{Example}
\newreptheorem{prop}{Proposition}
\newreptheorem{thm}{Theorem}

\begin{document}
\title{Ricci Flow on Torus Bundles}
\author{Dmytro Yeroshkin}\thanks{The author received funding from Excellence of Science grant number 30950721, ``Symplectic Techniques''}
\email{Dmytro.Yeroshkin@ulb.be}
\address{Geometri\'e Differentielle, Universit\'e Libre de Bruxelles}
\begin{abstract}
  In this paper we compute the Ricci flow formulas for invariant metrics on prinicpal $G$-bundles compatible with the connection. Our primary focus is on torus bundles which we use to study a notion of Bakry-\'Emery Ricci flow as well as Ricci flow on circle bundles over K\"ahler-Einstein manifolds. The latter application gives us solutions to Ricci flows on Heisenberg groups and implicit solutions to Ricci flows on Berger 3-spheres and several other 3-dimensional manifolds.
\end{abstract}

\maketitle

\makeatletter
\providecommand\@dotsep{5}
\makeatother

\section{Introduction}\label{sec:Intro}

In this paper we present computations for the behavior of Ricci flow on principal $G$-bundles $\pi:E\to M$ equipped with the metric
\[
    g_E = \pi^* g_M + Q(\mu,\mu)
\]
where $\mu$ is the principal connection and $Q$ is a family of right-invariant metrics on $G$ (we associate $\fg$ to $T_p F$ to allow us to compute $Q(\mu,\mu)$). Our main focus is on applications of this to Bakry-\'Emery Ricci curvature and to circle bundles over symplectic manifolds.

In section~\ref{sec:Background} we give a short overview of a few select aspects of Ricci flow. We also give a brief introduction to Bakry-\'Emery manifolds with density, focusing on the historical progress in that area.

The focus of section~\ref{sec:Principal} is to construct the necessary formulas. In subsection~\ref{sec:General} we compute the general Ricci flow formulas in the above setting. While the full generality formula is convenient to have, our main focus will be on the case of $G=T^q$. We provide the details of this restriction in subsection~\ref{sec:Torus}. It is worth pointing out that in certain ways the torus setting might be more natural, since we are able to view $Q$ as simply a family of inner products on $\fg$ parametrized by $p\in M$. In subsection~\ref{sec:flows}, we also point out the similarity between the formulas we have and several flows found in literature.

In section~\ref{sec:BE} we examine the details of flow in the Bakry-\'Emery case motivated by Ricci flow on the torus bundles. Including evolution formulas for various quantities and some discussion of singularity formation.

In section~\ref{sec:Circle} we further restrict our attention to circle bundles, and in particular the case of circle bundles over K\"ahler-Einstein manifolds. At first glance this structure might appear exceedingly restrictive, but as we show in subsection~\ref{sec:KE-Ex} we show that a lot of 3-dimensional model geometries can be realized in this way, as well as all left-invariant metrics on $(2n+1)$-dimensional Heisenberg groups.

\section{Background}\label{sec:Background}

\subsection{Ricci Flow}

Since its Introduction by Hamilton \cite{Hamilton}, Ricci flow, despite its simple formula:
\[
    \frac{\partial}{\partial t} g = -2\Ric(g),
\]
has proven to be an important tool in differential geometry. In particular, it played a key role in Perelman's proofs of the Poincar\'e conjectrue and Thurston's Geometrization conjectrue \cite{Per1,Per2,Per3}.

In general, solving Ricci flow requires solving a system of second-order elliptic PDEs, and as such can often not be solved. This issue can be overcome by requiring additional structure, such as a large isometry group, since Ricci flow must preserve isometry. An example of requiring a large isometry group is the work done on Ricci flow on homogeneous manifolds \cite{Lauret, BPRZ}.

In this paper we use a weaker isometry group restriction of a principal $G$-bundle structure, where we express the evolution of the metric on the total space in terms of the metric on the base manifold, the principal connection and a family of right-invariant metrics $Q$ on $G$.

One property that makes Ricci flow of interest particularly in the area of positively curved manifolds is the fact that for a compact manifold the minimum value of the scalar curvature is non-decreasing, since the evolution of scalar curvature is given by
\[
    \frac{\partial}{\partial t} R = \Delta R + 2|\Ric|^2.
\]
In section~\ref{sec:BE}, we present the analogue of this formula for the Bakry-\'Emery setting ($G=T^q$ with $Q=e^{-2f/q}g_{Eucl}$ and flat connection) entirely in terms of the metric on the base manifold and $f$.

\subsection{Bakry-\'Emery Manifodls with Density}

In section~\ref{sec:BE} we will consider the application of our calculations for Ricci flow in the context of Bakry-\'Emery manifolds with density. Motivated in part by the work of Lichnerowicz \cite{Lichnerowicz70,Lichnerowicz72}, Bakry and \'Emery \cite{BE} introduced a concept of modified Ricci curvature
\[
    \Ric_f^N = \Ric + \Hess f - \frac{df\otimes df}{N-n},
\]
where $N$ has traditionally been viewed as a parameter in $(n,\infty]$. One possible motivation for this definition is that for $N>n$ an integer, $\Ric_f^N$ is precisely the $M$ component of the Ricci curvature of the warped product metric $g = g_M + e^{2f/q}g_{flat}$ on $M\times T^q$, where $q=N-n$, as we will see in section~\ref{sec:BE}.

In the case of $N=\infty$, the parameter $N$ is often omitted from the notation, giving us $\Ric_f = \Ric + \Hess f$. This case is naturally associated with Ricci flow, since $\Ric_f = \lambda g$ is precisely the definition of a gradient Ricci soliton.

More recently, several authors have considered the case of $N<n$, see for example \cite{KolesnikovMilman,Milman1,Ohta,WSecDens}. In \cite{WSecDens}, Wylie constructed an analogue of sectional curvature compatible with Bakry-\'Emery Ricci curvature in the $N=\infty$ and $N=1$ cases, a construction that can be naturally expanded to other values of $N$. One motivation for considering specifically the more general case is the discovery in \cite{WYDens} by Wylie and the author of an affine connection $\nabla^f$ on $M$ whose Ricci curvature tensor is precisely $\Ric_f^1$. This connection is given by $\nabla^f_X Y = \nabla_X Y - \frac{df(X)}{n-1}Y - \frac{df(Y)}{n-1}X$, and was further studied in \cite{KWYDens,YHolonomy}. At present, no connection is known to give $\Ric_f^N$ in general for any other value of $N$.
\section{Ricci Flow on Principal $G$-Bundles}\label{sec:Principal}

\subsection{General Structrues}\label{sec:General}

Consider the principal bundle $G\to E\to M$. We use the usual convention of a right action of $G$ on $E$. We will require the metric on $E$ to be invariant under the $G$-action, which means that our metric on the fibers will be right-invariant. On the other hand, the natural formulation for the principal connection involves the Killing fields of the group action, as such we will use a left-invariant basis of vector fields along the fibers.

Let $g$ be a metric on $M$, $Q$ be a choice of a right invariant metric on the fibers, which may depend on the point $p\in M$, and $\mu$ be a principal connection. Then we construct the metric $\tilde{g}$ on $E$ given by $\tilde{g} = \pi^* g + \tr_Q \mu\otimes\mu$. Our convention throughout the paper will be to use Greek letters for indices along the base and Latin letters for indices along the fibers.

\begin{rmk}
Normally one would take $Q$ to be a family of inner products on $\fg$, but that is too restrictive since for principal bundles the convention is to have $G$ acting on the right, which leads to left-invariant Killing fields on the fibers. The only way to have a family of inner products that makes the $G$ act by isometries is to have $Q$ extend to a bi-invariant metric on $G$. Unfortunately, Ricci flow on the total space does not preserve this. For this reason we consider right-invariant metrics on the fibers instead.
\end{rmk}

Let $U\subset M$ be a coordinate chart and $s:U\to E$ be a section. We will utilize a convenient vector field basis for $\pi^{-1}(U)$. Let $\{E_i\}$ be a basis of left-invariant vector fields on $G$, using the right $G$-action, these extend naturally to a basis of vertical vector fields on $E$. Let $\{x^\beta\}$ be the local coordinates on $U$. Let $V_\beta = s_*\partial_\beta - \mu(s_*\partial_\beta)$ be the horizontal lift of the coordinate vector field $\partial_\beta$. We extend $V_\beta = (R_g)_* V_\beta$ to allow $V_\beta$ to exist beyond our chosen section.

We will write $c\indices{_{ij}^k}$ for the components of the Lie bracket, that is $[E_i,E_j] = c\indices{_{ij}^k}E_k$. We also have $[E_i,V_\beta] = 0$ and $[V_\beta,V_\gamma] = - F(\partial_\beta,\partial_\gamma) = - F\indices{^k_{\beta\gamma}}E_k$.

Note that the fact that our metric is right-invariant, but we are using a left-invariant vector field basis, we have $D_{E_i}Q_{jk} = c\indices{_{ij}^s}Q_{sk} + c\indices{_{ik}^s}Q_{sj} = c_{ijk} + c_{ikj}$.

In terms of our chosen vector field basis, we get $\tilde{g}_{\beta\gamma} = g_{\beta\gamma}$, $\tilde{g}_{ij} = Q_{ij}$ and $\tilde{g}_{i\gamma} = 0$.

We obtain the following formulas for the Levi-Civita connection of $\tilde{g}$ on $E$, which we denote $\tilde\nabla$:
\begin{align*}
\tilde\nabla_{E_i} E_j &= \nabla^\fg_{E_i} E_j - \frac{1}{2} g^{\lambda\nu}D_\nu(Q_{ij})V_\lambda\\
\tilde\nabla_{V_\gamma} E_i = \tilde\nabla_{E_i} V_\gamma &= \frac{1}{2}Q^{kl}D_\gamma(Q_{ik})E_l + \frac{1}{2}g^{\lambda\nu}Q_{is}F\indices{^s_{\gamma\nu}} V_\lambda\\
\tilde\nabla_{V_\beta} V_\gamma &= \Gamma_{\beta\gamma}^\nu V_\nu - \frac{1}{2}F_{\beta\gamma}
\end{align*}
where $D_\nu$ denotes the derivative in the $V_\nu$ direction, $\nabla^\fg$ is the Levi-Civita connection on $G$, and $\Gamma_{\beta\gamma}^\lambda$ is the Christoffel symbol on $M$.

The corresponding Ricci curvatures are:
\begin{align*}
\tilde\Ric_{jk} &= \Ric^\fg_{jk} - \frac{1}{4}g^{\tau\nu}Q^{su}\left(D_\tau Q_{su}\right)\left(D_\nu Q_{jk}\right) + \frac{1}{4}g^{\lambda\nu}g^{\tau\xi}Q_{js}Q_{ku}F\indices{^s_{\lambda\tau}}F\indices{^u_{\nu\xi}}\\
&\qquad - \frac{1}{2}D_\lambda\left(g^{\lambda\nu}D_\nu Q_{jk}\right) - \frac{1}{2}\Gamma_{\lambda\tau}^\lambda g^{\tau\nu}D_\nu Q_{jk} + \frac{1}{2}Q^{su}g^{\lambda\nu}\left(D_\nu Q_{js}\right)\left(D_\lambda Q_{ku}\right)\\
\tilde\Ric_{j\gamma} &= \frac{1}{2}c\indices{^s_j^m}D_\gamma Q_{sm} - \frac{1}{2}c\indices{_u^{su}}D_\gamma Q_{js} + \frac{1}{4}Q_{jl}g^{\tau\nu}F\indices{^l_{\gamma\nu}} Q^{um}  D_\tau Q_{um}\\
&\qquad + \frac{1}{2} D_\lambda\left(g^{\lambda\nu}Q_{js}F\indices{^s_{\gamma\nu}}\right) + \frac{1}{2}\Gamma_{\lambda\tau}^\lambda g^{\tau\nu}Q_{js}F\indices{^s_{\gamma\nu}} - \frac{1}{2}g^{\lambda\nu}\Gamma_{\lambda\gamma}^\tau Q_{js}F\indices{^s_{\tau\nu}}\\
\tilde\Ric_{\beta\gamma} &= \Ric^M_{\beta\gamma} - \frac{1}{2}g^{\lambda\nu}Q_{su}F\indices{^s_{\beta\lambda}}F\indices{^u_{\gamma\nu}} - \frac{1}{4}\left( D_\beta Q^{us}\right)\left( D_\gamma Q_{us}\right)\\
&\qquad - \frac{1}{2}Q^{us} D_\beta D_\gamma Q_{us} + \frac{1}{2}\Gamma_{\beta\gamma}^\tau Q^{us} D_\tau Q_{us} + \frac{1}{2}c\indices{_{us}^u}F\indices{^s_{\beta\gamma}}
\end{align*}

\begin{rmk}
One might be concerned about the presence of $F\indices{^s_{\beta\gamma}}$, a skew-symmetric term, in the expression for $\tilde\Ric_{\beta\gamma}$, which is symmetric. However, this term is necessary to cancel out the skew-symmetric part of the $D_\beta D_\gamma Q_{us}$ term, as we show below (note we double the terms for convenience).
\begin{align*}
&-Q^{us} D_\beta D_\gamma Q_{us} + Q^{us}c_{ums}F\indices{^m_{\beta\gamma}} + Q^{us} D_\gamma D_\beta Q_{us} - Q^{us}c_{ums}F\indices{^m_{\gamma\beta}}\\
&\qquad\qquad = Q^{us}\left[D_{[V_\gamma,V_\beta]} Q_{us} + 2c_{ums}F\indices{^m_{\beta\gamma}}\right]\\
&\qquad\qquad = Q^{us}\left[-F\indices{^m_{\gamma\beta}}D_{E_m} Q_{us} + 2c_{ums}F\indices{^m_{\beta\gamma}}\right]\\
&\qquad\qquad = Q^{us}F\indices{^m_{\beta\gamma}}\left[c_{mus} + c_{msu} - 2c_{mus}\right]\\
&\qquad\qquad = 0
\end{align*}
\end{rmk}

\begin{prop}
Under Ricci flow of $(E,\tilde{g})$, provided it is well-defined, the triple $(g,Q,\mu)$ evolves by:
\begin{align*}
\partial_t g_{\beta\gamma} &= -2\tilde\Ric_{\beta\gamma}\\
\partial_t Q_{jk} &= -2\tilde\Ric_{jk}\\
\partial_t \mu\indices{^k_\gamma} &= -2Q^{jk}\tilde\Ric_{j\gamma}
\end{align*}
Here $\mu\indices{^k_\gamma}$ denotes the $E_k$ component of $\mu(s_*\partial_\gamma)$, where $s:U\to E$ is our local section.
\end{prop}

\begin{proof}
Our first observation is the fact that the vector fields $E_i$ are defined purely in terms of the $G$-action, and as such do not change as we evolve the metric.

Recall that $Q_{jk} = Q(E_j,E_k) = \tilde{g}_{jk}$, so $\partial_t Q_{jk} = -2\tilde\Ric_{jk}$ by the definition of Ricci flow.

Next we consider $\partial_t\mu$. In order to compute this, we must examine how the horizontal distribution changes. Consider $\partial_\gamma$ a coordinate vector field along $U\subset M$, then along $s(U)$, $V_\gamma = s_*\partial_\gamma - \mu(s_*\partial_\gamma) = s_*\partial_\gamma - \mu\indices{^k_\gamma}E_k$, the only part of this that can evolve with time is $\mu$. Now, consider
\begin{align*}
\partial_t\Big[\tilde{g}(E_j,V_\gamma)\Big] &= 0\\
&= (\partial_t \tilde{g})(E_j,V_\gamma) + \tilde{g}(E_j,\partial_t V_\gamma)\\
&= -2\tilde\Ric_{j\gamma} - Q_{jk}\partial_t \mu\indices{^k_\gamma}
\end{align*}
So, $\partial_t\mu\indices{^k_\gamma} = -2Q^{jk}\tilde\Ric_{j\gamma}$ as claimed.

Finally, observe that $g_{\beta\gamma} = g(\partial_\beta,\partial_\gamma) = \tilde{g}(s_*\partial_\beta, s_*\partial_\gamma) - Q_{kl}\mu\indices{^k_\beta}\mu\indices{^l_\gamma}$ and $\tilde\Ric(s_*\partial_\beta, s_*\partial_\gamma) = \tilde\Ric_{\beta\gamma} + \mu\indices{^k_\beta}\tilde\Ric_{k\gamma} + \mu\indices{^l_\gamma}\tilde\Ric_{\beta l} + \mu\indices{^k_\beta}\mu\indices{^l_\gamma}\tilde\Ric_{kl}$. Putting these together, we obtain:

\begin{align*}
\partial_t g_{\beta\gamma} &= \partial_t\Big[\tilde{g}(s_*\partial_\beta, s_*\partial_\gamma) - Q_{kl}\mu\indices{^k_\beta}\mu\indices{^l_\gamma}\Big]\\
&= (\partial_t \tilde{g})(s_*\partial_\beta, s_*\partial_\gamma) - \partial_t(Q_{kl}\mu\indices{^k_\beta}\mu\indices{^l_\gamma})\\
&= -2\tilde\Ric(s_*\partial_\beta, s_*\partial_\gamma) - (\partial_t Q_{kl}) \mu\indices{^k_\beta}\mu\indices{^l_\gamma} - Q_{kl} (\partial_t \mu\indices{^k_\beta})\mu\indices{^l_\gamma} - Q_{kl}\mu\indices{^k_\beta}\partial_t\mu\indices{^l_\gamma}\\
&= -2\tilde\Ric_ {\beta\gamma} - 2\mu\indices{^k_\beta}\tilde\Ric_{k\gamma} - 2\mu\indices{^l_\gamma}\tilde\Ric_{\beta l} - 2\mu\indices{^k_\beta}\mu\indices{^l_\gamma}\tilde\Ric_{kl}\\
&\qquad + 2\tilde\Ric_{kl}\mu\indices{^k_\beta}\mu\indices{^l_\gamma} + 2 Q_{kl}Q^{ks}\tilde\Ric_{s\beta}\mu\indices{^l_\gamma} + 2Q_{kl}\mu\indices{^k_\beta}Q^{ls}\tilde\Ric_{s\gamma}\\
&= -2\tilde\Ric_{\beta\gamma}
\end{align*}
\end{proof}

While we do not make further use of the fully general formulas, we make a note of one possible future application. In particular, if one is interested in searching for Einstein manifolds with high degree of symmetry, one might ask that the manifold be the total space of a principal $G$-bundle. Since $\tilde{g}_{j\gamma}=0$, the first restriction one obtains under these assumptions is that $\tilde\Ric_{j\gamma} = 0$.

\subsection{Torus Bundles}\label{sec:Torus}

For most of this paper, we will be primarily interested in the setting where $G = T^q$. The obvious changes this imposes compared to the general case are that $c\indices{_{ij}^k}\equiv 0$ and $\Ric^\fg \equiv 0$. Furthermore, since any right-invariant metric on $T^q$ is automatically bi-invariant, $\partial_{E_i} Q_{jk} = 0$, and in particular, we may view $Q:M\to \mathrm{S}_+^2(\bbR)$ (the space of positive-definite symmetric $q\times q$ matrices) and furthermore, since $c\equiv 0$, we get $D_{V_\beta} =  D_\beta = D_{s_*\partial_\beta}$. Under these simplifications, we obtain:
\begin{align*}
\tilde\Ric_{jk} &= -\frac{1}{2}\Delta Q_{jk} + \frac{1}{2}Q^{su}(\nabla^\lambda Q_{js})\nabla_\lambda Q_{ku} - \frac{1}{4}Q^{su}(\nabla^\lambda Q_{su})\nabla_\lambda Q_{jk}\\
&\qquad + \frac{1}{4}g^{\lambda\nu}g^{\tau\xi}Q_{js}Q_{ku} F\indices{^s_{\lambda\tau}}F\indices{^u_{\nu\xi}}\\
\tilde\Ric_{j\gamma} &= \frac{1}{2}\nabla^\lambda \left(Q_{js}F\indices{^s_{\gamma\lambda}}\right) + \frac{1}{4}Q_{jl}Q^{um}F\indices{^l_{\gamma\lambda}}\nabla^\lambda Q_{um}\\
\tilde\Ric_{\beta\gamma} &= \Ric^M_{\beta\gamma} - \frac{1}{2}Q^{us}\nabla_\beta\nabla_\gamma Q_{us} - \frac{1}{4}(\nabla_\beta Q^{us})\nabla_\gamma Q_{us} - \frac{1}{2}g^{\lambda\nu}Q_{su}F\indices{^s_{\beta\gamma}}F\indices{^u_{\gamma\nu}}
\end{align*}
Here $\Delta,\nabla$ denote the laplacian and covariant derivatives along $M$.

The resulting evolution formulas are:
\begin{align*}
	\frac{\partial}{\partial t} g_{\beta\gamma} &= -2\Ric_{\beta\gamma} + \frac{1}{2}\left(\nabla_\beta Q^{ks}\right)\nabla_\gamma Q_{ks} + Q^{ks}\nabla_\beta\nabla_\gamma Q_{ks} + g^{\lambda\nu}Q_{iu} F\indices{^u_{\beta\lambda}}F\indices{^i_{\gamma\nu}}\\
	\frac{\partial}{\partial t} \mu\indices{^i_\gamma} &= \nabla^\lambda F\indices{^i_{\lambda\gamma}} + \frac{1}{2}F\indices{^i_{\lambda\gamma}}Q^{su}\nabla^\lambda Q_{su} + F\indices{^l_{\lambda\gamma}}Q^{is}\nabla^\lambda Q_{ls}\\
	\frac{\partial}{\partial t} Q_{jk} &= \Delta Q_{jk} - Q^{su}\left(\nabla_\tau Q_{sj}\right)\left(\nabla^\tau Q_{uk}\right) + \frac{1}{2}Q^{su}\left(\nabla_\tau Q_{su}\right)\left(\nabla^\tau Q_{jk}\right)\\
	&\qquad - \frac{1}{2}g^{\lambda\nu}g^{\tau\xi}Q_{sj}Q_{uk}F\indices{^s_{\lambda\tau}}F\indices{^u_{\nu\xi}}
\end{align*}
These formulas can be found in \cite{LottTorii}.

We will largely be interested in what happens with a chosen section, in which case we can let $\alpha = s^*\mu$, which gives $F = d\alpha$. The evolution formulas for $(g,\alpha,Q)$ are:
\begin{align*}
	\frac{\partial}{\partial t} g_{\beta\gamma} &= -2\Ric_{\beta\gamma} + \frac{1}{2}\left(\nabla_\beta Q^{ks}\right)\nabla_\gamma Q_{ks} + Q^{ks}\nabla_\beta\nabla_\gamma Q_{ks} + g^{\mu\nu}Q_{iu} d\alpha\indices{^u_{\beta\mu}}d\alpha\indices{^i_{\gamma\nu}}\\
	\frac{\partial}{\partial t} \alpha\indices{^i_\gamma} &= \nabla^\mu d\alpha\indices{^i_{\mu\gamma}} + \frac{1}{2}d\alpha\indices{^i_{\mu\gamma}}Q^{su}\nabla^\mu Q_{su} + d\alpha\indices{^l_{\mu\gamma}}Q^{is}\nabla^\mu Q_{ls}\\
	\frac{\partial}{\partial t} Q_{jk} &= \Delta Q_{jk} - Q^{su}\left(\nabla_\mu Q_{sj}\right)\left(\nabla^\mu Q_{uk}\right) + \frac{1}{2}Q^{su}\left(\nabla_\mu Q_{su}\right)\left(\nabla^\mu Q_{jk}\right)\\
	&\qquad - \frac{1}{2}g^{\lambda\nu}g^{\mu\xi}Q_{sj}Q_{uk}d\alpha\indices{^s_{\lambda\mu}}d\alpha\indices{^u_{\nu\xi}}.
\end{align*}

\subsection{Related Flows}\label{sec:flows}

We want to raise similarities between this flow and 2 other flows found in literature: Ricci-Yang-Mills flow by Jeffrey Streets \cite{StreetsThesis,StreetsRYM} and hypersymplectic flow by Fine and Yao \cite{FYhsymp}. \todo{Add Harmonic Ricci flow}

The Ricci-Yang-Mills flow by Streets has a similar set-up as our flow, with the added restriction that $Q$ is invariant both with respect to the choice of $p\in M$ and with respect to time. The resulting flow differs by adding a coefficient of $2$ to the last term in the formula for $\partial_t g$.

The hypersymplectic flow by Fine and Yao starts from a triple of symplectic forms $\omega_1,\omega_2,\omega_3$ on a 4-manifold which satisfy $\omega_i\wedge\omega_j=A_{ij}\mu$ where $\mu$ is a chosen volume form and $A$ is a poisitive definite matrix. $g$ and $Q$ are constructed from these forms, with the constraint of $\det Q = 1$. The evolution itself is obtained by constructing a $G_2$-form $\phi$ on $M\times T^3$ and evolving $\phi$ by $\Delta\phi$. If one selects $\mu$ such that $F^i = \omega_i$, then the evolution formulas for $g$ and $Q$ have the same leading terms, and the evolution formula for $\mu$ is the negative of the one we obtained.
\todo{Expand the details of both of these evolutions for a more thorough comparison}
\section{Bakry-\'Emery Ricci Flow}\label{sec:BE}

One application of the principal bundle Ricci flow is Bakry-\'Emery Ricci flow. The setup we use for this is $Q_{ij} = e^{-2f/q}\delta_{ij}$, $F=0$ and $N = n + q$. Note that $F=0$ is equivalent to an existence of a horizontal section around each point, we will use these sections to work in local coordinates. Then we get:

\begin{align*}
\frac{\partial}{\partial t} g &= -2\Ric - 2 \Hess\ f + \frac{2}{q} df\otimes df  = -2\Ric_f^N\\
\frac{\partial}{\partial t} f &= \Delta f - |\nabla f|^2 = \Delta_f f
\end{align*}
Here $\Ric_f^N$ denotes the Bakry-\'Emery Ricci tensor, and $\Delta_f$ is the drift Laplacian $\Delta_f u = \Delta u - df(\nabla u)$. Note that these formulas were presented in \cite{LottBE}.

We observe that while this formula comes from Ricci flow in the case where $N>n$ is an integer, the evolution formulas are well defined (under standard Ricci flow assumptions, e.g. when $M$ is compact) for any $N\neq n$ since they would differ only by lower order terms.

\begin{prop}
Under the flow
\begin{align*}
\frac{\partial}{\partial t} g &= -2\Ric_f^N\\
\frac{\partial}{\partial t} f &= \Delta_f f
\end{align*}
with $N\neq n$, we have the following evolution formulas:
    \begin{align*}
	\frac{\partial}{\partial t} \Ric_{jk} &= \Delta\Ric_{jk} + 2R_{pjkq}\Ric^{pq} - 2\Ric_j^p\Ric_{pk}\\
	&\qquad - \left(\nabla^\lambda\Ric_{jk}\right)\nabla_\lambda f - \Ric_j^\lambda\nabla_k\nabla_\lambda f - \Ric_k^\lambda\nabla_j\nabla_\lambda f\\
	&\qquad +\frac{2}{N-n}\left[\Delta f \nabla_j\nabla_k f - \nabla_p\nabla_j f\nabla^p\nabla_k f - R_{pjk}^\lambda\nabla_\lambda f\nabla^p f\right]\\
	\frac{\partial}{\partial t}\bar{\Ric}_{jk} &= \Delta\bar{\Ric}_{jk} + 2R_{pjkq}\bar{\Ric}^{pq} -2\bar{\Ric}_j^p\bar{\Ric}_{pk} - \left(\nabla^p \bar{\Ric}_{jk}\right)\nabla_p f\\
	&\qquad - \frac{1}{N-n}\Big[\bar{\Ric}_{jp}\nabla^p f\nabla_k f + \bar{\Ric}_{kp}\nabla^p f\nabla_j f - 2\left(\Delta f - |\nabla f|^2\right) \nabla_j\nabla_k f \Big]\\&= \Delta_f\bar{\Ric}_{jk} + 2R_{pjkq}\bar{\Ric}^{pq} -2\bar{\Ric}_j^p\bar{\Ric}_{pk}\\
	&\qquad - \frac{1}{N-n}\Big[\bar{\Ric}_{jp}\nabla^p f\nabla_k f + \bar{\Ric}_{kp}\nabla^p f\nabla_j f - 2\left(\Delta f - |\nabla f|^2\right) \nabla_j\nabla_k f \Big]\\
	\frac{\partial}{\partial t} \bar{S} &= \Delta \bar S + 2|\bar\Ric|^2 - 2\langle \bar{\Ric}, \Hess f\rangle - df(\nabla\bar S) + \frac{2}{N-n}\Delta f\Big[\Delta f - |\nabla f|^2\Big]\\
	&= \Delta_f \bar{S} + 2|\bar\Ric|^2 - 2\langle\bar\Ric,\Hess f\rangle + \frac{2}{N-n}\Delta f\Big[\Delta f - |\nabla f|^2\Big]
\end{align*}
where $\bar\Ric = \Ric_f^N$ and $\bar{S} = g^{ij}\bar\Ric_{ij}$ is the Bakry-\'Emery scalar curvature.

Furthermore, if we consider $\tilde{S}_k = \bar{S} + \Delta_f f - k|\nabla f|^2 = \bar{S} + \Delta f - (k+1)|\nabla f|^2$, we get

\[
\frac{\partial}{\partial t}\tilde{S}_k = \Delta_f\tilde{S}_k + 2|\bar\Ric|^2 + \frac{2}{N-n}\Big[\Delta f - |\nabla f|^2\Big]^2 + 2k|\Hess f|^2 + \frac{2k}{N-n}|\nabla f|^4
\]
\end{prop}

\begin{prop}
Let $(M^n,g,f,N)$ be a compact manifold with density and $N\in(n,\infty]$. Furthermore, suppose that Bakry-\'Emery Ricci flow exists for $t\in[0,T)$. Then, $\inf \tilde{S}_k$ is non-decreasing for any $k\geq 0$.
\end{prop}

\begin{rmk}
	For $N>n$ an integer, $\tilde{S}_0$ is precisely the scalar curvature of the bundle
\end{rmk}

\begin{thm}
	Let $(M^n,g,f,N)$ be a manifold with density whose Bakry-\'Emery Ricci flow exists on $[0,T)$. If $|\nabla f|^2\leq k$ is bounded, then, we have the following bounds on $|\nabla f_t|^2$:
	\[
	|\nabla f_t|^2 \leq \begin{cases}
		k & N\in(n,\infty]\\
		\dfrac{(n-N)k}{(n-N)-2k t} & N\in(-\infty,n)
	\end{cases}
	\]
\end{thm}

\begin{cor}
	Let $(M^n,g,f,N)$ be a manifold with density and $N\in(n,\infty]$ whose Bakry-\'Emery Ricci flow exists on $[0,T)$. If both $f$ and $|\nabla f|$ are bounded at $t=0$, then either the flow is immortal or it has a purely metric singularity.
\end{cor}


\todo{An example with variable $N$}

\section{Ricci Flow on Circle Bundles}\label{sec:Circle}

The next special case we consider is when $G=S^1$ (alternatively we can look at the universal cover $G=\bbR$). For convenience, we normalize $Q = \begin{pmatrix}e^{-2f}\end{pmatrix}$, similar to the Bakry-\'Emery setting.

\begin{align*}
\frac{\partial}{\partial t} g_{\beta\gamma} &= -2\Ric_{\beta\gamma} - 2\nabla_\beta f\nabla_\gamma f + e^{-2f}g^{\lambda\nu} F_{\beta\lambda}F_{\gamma\nu}\\
\frac{\partial}{\partial t} f &= \Delta_f f + \frac{1}{4} e^{-2f}g^{\beta\gamma}g^{\lambda\nu}F_{\beta\lambda}F_{\gamma\nu}\\
\frac{\partial}{\partial t} \mu_\beta &= \nabla^\lambda F_{\lambda\beta} - 3\left(\nabla^\lambda f\right) F_{\lambda\beta}\\
\frac{\partial}{\partial t} \mu &= -d^\star F - 3\imath_{\nabla f} F\\
\frac{\partial}{\partial t} F &= -\Delta^H F - 3d\left(\imath_{\nabla f} F\right)
\end{align*}
Here we view $\mu,F$ as real valued forms by considering only the coefficient of the chosen generator of $\fg$. The other observation we make is that the $g^{\beta\gamma}g^{\lambda\nu}F_{\beta\lambda}F_{\gamma\nu}$ term in the evolution of $f$ is either $2|F|^2$ or $\frac{1}{2}|F|^2$ depending on the convention one uses.

One case of particular interest is when $(M,g_0,\omega)$ is a K\"ahler-Einstein manifold with the curvature of the circle bundle $F = \omega$, and $f\equiv C$. For convenience, we will allow $g$ to be a scalar multiple of $g_0$.

\begin{thm}\label{thm:KE}
Let $(M,g_0,\omega)$ be a K\"ahler-Einstein manifold of complex dimension $n$ with $\Ric(g_0) = \lambda g_0$. Construct a principal circle bundle over $(M,u_0g_0)$ with curvature $F=\omega$ and constant fiber volume ($f\equiv C$). Then the Ricci flow on the bundle is given by $g(t) = u(t)g_0$, $f_t = f(t)$ and the principal connection is preserved. Furthermore, $u(t),f(t)$ are the solutions to the following system of ODEs:
\begin{align*}
u'(t) &= -2\lambda + \frac{e^{-2f}}{u} & u(0) &= u_0\\
f'(t) &= \frac{ne^{-2f}}{2u^2} & f(0) &= C
\end{align*}
Furthermore, if $\lambda=0$, we get
\begin{align*}
u(t) &= e^{-2C/(n+2)}u_0^{n/(n+2)}\left[(n+2)t + e^{2C}u_0^2\right]^{1/(n+2)}\\
f(t) &= \frac{n}{2(n+2)}\log\left((n+2)t + e^{2C}u_0^2\right) + \frac{2C}{n+2} - \frac{n}{n+2}\log(u_0)
\end{align*}
If $\lambda\neq 0$, then
\[
\Psi(t) = e^{2f}\left(1 - \frac{n+1}{2\lambda u e^{2f}}\right)^{n/(n+1)}
\]
is constant.
\end{thm}

\begin{proof}
The above formulas are easy to verify by hand.

The construction of $\Psi(t)$ stems from the observation that
\begin{align*}
\frac{\partial}{\partial t} \left(u^{1-k}e^{-2kf}\right) &= (1-k)\left(-2\lambda + \frac{e^{-2f}}{u}\right)u^{-k}e^{-2kf} + -2k\left(\frac{ne^{-2f}}{2u^2}\right)u^{1-k}e^{-2kf}\\
&= -2\lambda(1-k)u^{-k}e^{-2kf} + \left((1-k) - kn\right)u^{-(k+1)}e^{-2(k+1)f}
\end{align*}
This allows one to choose coefficients of $\displaystyle\sum_{k=1}^\infty C_k u^{1-k}e^{-2kf}$ such that the derivative telescopes. If the resulting series converges, the limit is a constant multiple of $\dfrac{1}{\Psi(t)}$.
\end{proof}

It is perhaps surprising how many examples are constructed in the fashion of Theorem~\ref{thm:KE}. In Section~\ref{sec:KE-Ex} we will construct examples in the setting of several Lie groups: $SU(2),\ SL_2(\bbR)$ and one example of a solvable Lie group in dimension 3, as well as the arbitrary dimensional Heisenberg groups.

We note that the evolution formulas in Theorem~\ref{thm:KE} are equivalent to those in \cite[Section~3.4]{Lauret} in the 3-dimensional case. In that paper Lauret considers 3-dimensional homogeneous manifolds, and in particular in that section he examines the case where the eigenvalues of the Ricci operator are
\[
\frac{a^2}{2},\ b-\frac{a^2}{2},\ b-\frac{a^2}{2}.
\]
In our construction, the eigenvalues of the Ricci operator are
\[
\frac{ne^{-2f}}{2u^2},\ \underbrace{\frac{\lambda}{u} - \frac{e^{-2f}}{2u^2},\ldots,\ \frac{\lambda}{u} - \frac{e^{-2f}}{2u^2}}_{2n\ \text{copies}}
\]
This provides the immediate equivalence under the identification of
\[
a = \frac{e^{-f}}{u},\ b = \frac{\lambda}{u}\\
\]
Lauret shows that the evolution of $a$ and $b$ is given by the ODE system:
\begin{align*}
a' &= \left(2b-\frac{3}{2}a^2\right)a\\
b' &= \left(2b-a^2\right)b
\end{align*}
We note that Lauret's formulation can handle the case of $a=0$, while our formulation cannot. This is easy to remedy, as the $a=0$ case can be constructed by taking $g_0$ to be Einstein with $\Ric(g_0) = \lambda g_0$, $g=u(t)g_0$, $f\equiv C$ and $F=0$. This still maintains $b = \frac{\lambda}{u}$, and we can easily solve the resulting ODE: $u(t) = u_0 -  2\lambda t$.

The other observation we make is that in Lauret's notation we can construct an analogue of $\Psi$:
\[
	\Lambda(t) = \lambda^4\Psi(t)^2 = \frac{b^4}{a^4} - \frac{b^3}{a^2}
\]
which is also constant, as can be verified directly from Lauret's formulas.

\subsection{Examples}\label{sec:KE-Ex}

In this section we demonstrate some examples of Ricci flow computations that can be handled using the results of Theorem~\ref{thm:KE}

\begin{ex}[Berger 3-Spheres]\label{ex:S3Berger}
Recall that a Berger 3-sphere is $SU(2)$ equipped with a metric that is left-invariant and right $SO(2)$-invariant. Equivalently, this is the left-invariant metric given by $Q(I,I) = \lambda_1^2$, $Q(J,J) = Q(K,K) = \lambda_2^2$ and $Q(I,J) = Q(I,K) = Q(J,K) = 0$, in this notation we have $SO(2) = \{\exp(tI)|t\in\bbR\}$ acting isometrically on the right.

Recall that
\[
	SU(2) = \left\{\begin{pmatrix}
		z & w\\
		-\bar{w} & \bar{z}
	\end{pmatrix}:z,w\in\bbC,\ |z|^2+|w|^2 = 1\right\}
\]
and our chosen $SO(2)$ corresponds to the matrices with $w=0$.

Quotienting out by the $SO(2)$ action, we get a map $\phi:SU(2)\to S^2$, which can be extended to a map $\tilde\phi:\bbC^2\to\bbR\times\bbC\cong\bbR^3$ given by $\tilde\phi(z,w) = (|z|^2-|w|^2, 2\bar{z}w)$. Note that the fact that our metric is left-invariant ensures that $S^2$ has a homogeneous metric, and looking at the image of the identity, we have $I\mapsto(0,0),\ J\mapsto(0,2),\ K\mapsto(0,2i)$, which means that our $S^2$ ends up with $g = (\lambda_2^2/4)g_{unit}$, where $g_{unit}$ is the standard metric on the round unit sphere. We also get $F(\phi_*J, \phi_*K) = -2$, so reversing the orientation, we have $F = \frac{1}{2}\omega_0$, where $\omega_0$ is the canonical volume form on the unit $S^2$.

To have $F = \omega$, we need to set $g_0 = \frac{1}{2}g_{unit}$, which gives us $\lambda = 2$, and our initial conditions then are $u(0) = \lambda_2^2/2$ and $e^{-2f(0)} = \lambda_1^2$. This gives us the following ODE system:
\begin{align*}
u'(t) &= -4 + \frac{e^{-2f}}{u} & u(0) &= \frac{\lambda_2^2}{2}\\
f'(t) &= \frac{e^{-2f}}{2u^2} & f(0) &= -\log(\lambda_1)
\end{align*}
In the special case of $\lambda_1=\lambda_2$ we can solve this system as $u(t) = -2t+\frac{\lambda_2^2}{2}$ and $e^{-2f(t)} = -4t+\lambda_2^2$, which is the correct flow for the round 3-sphere.

If $\lambda_1\neq\lambda_2$, we get an implicit solution:
\[
	e^{2f}\left(1-\frac{1}{2ue^{2f}}\right)^{1/2} = \frac{\sqrt{\lambda_2^2-\lambda_1^2}}{\lambda_1^2\lambda_2}
\]
We plot some examples of this flow using $x=e^{-2f}, y=u$ in Figure~\ref{fig:Berger}, where the legend provides the initial values of $\lambda_1,\lambda_2$. In all cases the flow predictibly collapses to $g=0$, with asymptotic behavior of a round sphere.

\begin{figure}\label{fig:Berger}
\includegraphics[width=\textwidth]{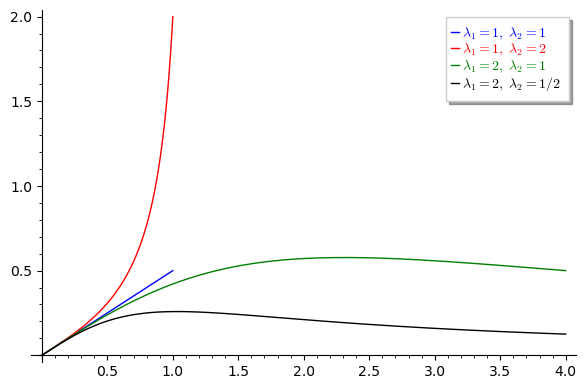}

\caption{Ricci flow on Berger spheres}
\end{figure}
\end{ex}

\begin{ex}[$SL_2(\bbR)$]\label{ex:SL2R}
The analogues of Berger metrics on $SL_2(\bbR)$ are the left-invariant, right $SO(2)$-invariant metrics. These metrics are of the form $Q(U,U)=\lambda_1^2$, $Q(V,V)=Q(W,W)=\lambda_2^2$ and $Q(U,V)=Q(U,W)=Q(V,W)=0$, where
\[
U = \begin{pmatrix}
	0 & 1\\
	-1 & 0
\end{pmatrix}\qquad
V = \begin{pmatrix}
	0 & 1\\
	1 & 0
\end{pmatrix}\qquad
W = \begin{pmatrix}
	1 & 0\\
	0 & -1
\end{pmatrix}.
\]

Under right-action of $SO(2)$, the quotient manifold can be naturally identified with the upper half-plane: $(x,y)\sim \begin{pmatrix}
y & x/y\\
0 & 1/y
\end{pmatrix}$. Furthermore, the induced metric on the quotient manifold is
\[
g = \begin{pmatrix}
\frac{\lambda_2^2}{4y^4} & 0\\
0 & \frac{\lambda_2^2}{y^2}
\end{pmatrix}.
\]
For convenience, we utilize $g_0 = \begin{pmatrix}
	\frac{1}{4y^3} & 0\\
	0 & \frac{1}{y^2}
\end{pmatrix}$, since $2 = F(V,W) = 4y^3\omega(\partial_x,\partial_y)$, which gives us the correct $\omega$ for $g_0$. We also get $\lambda = -4$, $u(0) = \lambda_2^2$ and $f(0) = -\log(\lambda_1)$.

The resulting implicit solution is
\[
e^{2f}\left(1+\frac{1}{4ue^{2f}}\right)^{1/2} = \frac{\sqrt{\lambda_1^2+4\lambda_2^2}}{2\lambda_1^2\lambda_2}.
\]
\end{ex}

\begin{ex}[Heisenberg Group]\label{ex:Heisenberg}
Let $H_{2n+1}$ be the $(2n+1)$-dimensional Heisenberg group, and $\fh_{2n+1}$ the corresponding Lie algebra.

We consider the standard model for the Lie group:
\begin{align*}
H_{2n+1} &= \left\{
\begin{pmatrix}
	1 & x_1 & \cdots & x_n & z\\
	0 & 1 & \cdots & 0 & y_1\\
	\vdots & & \ddots & \vdots & \vdots\\
	0 & 0 & \cdots & 1 & y_n\\
	0 & 0 & \cdots & 0 & 1
\end{pmatrix} :
x_i,y_i,z\in\bbR
\right\}\\
\fh_{2n+1} &= \left\{\begin{pmatrix}
	0 & u_1 & \cdots & u_n & w\\
	0 & 0 & \cdots & 0 & v_1\\
	\vdots & & \ddots & \vdots & \vdots\\
	0 & 0 & \cdots & 0 & v_n\\
	0 & 0 & \cdots & 0 & 0
\end{pmatrix} :
u_i,v_i,w\in\bbR
\right\}
\end{align*}
Consider the basis $X_i,Y_i,Z$ on $\fh_{2n+1}$ with $X_i$ corresponding to $u_i=1$, $Y_i$ to $v_i=1$ and $Z$ to $w=1$ (all other entries zero). Then, the left-invariant extensions in coordinates are
\[
X_i = \partial_{x_i},\qquad Y_i = \partial_{y_i} + x_i\partial_z,\qquad Z = \partial_z
\]
The only non-zero brackets are $[X_i,Y_i] = Z$.

Using a correct choice of basis, we can ensure that the inner product $Q$ satisfies
\begin{align*}
Q(X_i,X_j) = Q(Y_i,Y_j) &= \delta_{ij}\\
Q(X_i,Y_j) = Q(X_i,Z) = Q(Y_j,Z) &= 0\\
Q(Z,Z) &= c^2
\end{align*}

As a left-invariant metric on $H_{2n+1}$ we can write this in coordinates as
\begin{align*}
	g(\partial_{x_i},\partial_{x_j}) &= \delta_{ij}\\
	g(\partial_{x_i},\partial_{y_j}) = g(\partial_{x_i},\partial_z) &= 0\\
	g(\partial_{y_i},\partial_{y_j}) &= \delta_{ij} + x_i x_j c^2\\
	g(\partial_{y_i},\partial_z) &= -c^2 x_i\\
	g(\partial_z,\partial_z) &= c^2
\end{align*}

Letting $z$ be the coordinate along the fiber, in our notation this gives us
\begin{align*}
	g &= g_{Eucl} & f &= -\log c\\
	\mu &= -\sum_{i=1}^n x_i dy_i + dz & F &= -\sum_{i=1}^n dx_i\wedge dy_i
\end{align*}

Since this is a K\"ahler Ricci-flat example, we know the explicit solution:

\begin{align*}
	u(t) &= c^{2/(n+2)}\left[(n+2)t+\frac{1}{c^2}\right]^{1/(n+2)}\\
	f(t) &= \frac{n}{2(n+2)}\log\left((n+2)t + \frac{1}{c^2}\right) - \frac{2\log c}{n+2}
\end{align*}

One interesting observation is that if we continuously adjust our coordinates to have a metric on $\fh_{2n+1}$ with the same conditions as we constructed above:
\begin{align*}
	g(\partial_{x_i},\partial_{x_j}) &= \delta_{ij}\\
	g(\partial_{x_i},\partial_{y_j}) = g(\partial_{x_i},\partial_z) &= 0\\
	g(\partial_{y_i},\partial_{y_j}) &= \delta_{ij} + x_i x_j c(t)^2\\
	g(\partial_{y_i},\partial_z) &= -c(t)^2 x_i\\
	g(\partial_z,\partial_z) &= c(t)^2
\end{align*}
we get
\[
c(t) = \frac{1}{\sqrt{u^2 e^{2f}}} = \frac{1}{\sqrt{(n+2)t + c^{-2}}} = \frac{c}{\sqrt{1+(n+2)c^2 t}}
\]
\end{ex}


The last example we want to present here is one of the $Sol$ manifolds. In particular, we will be looking at the Lie group corresponding to the solvable Lie algebra of type III under Bianchi classification \cite{BianchiIta,BianchiEng}. This Lie algebra is the direct sum of the unique non-abelian 2-dimensional Lie algebra and $\bbR$.

\begin{ex}[Sol III]\label{ex:Sol3}
Let $\fg$ be the type III Lie algebra. Recall that the Lie bracket of $\fg$ is given by $[X,Y] = Y,\ [Z,-]=0$.

Let $Q$ be any inner product on $\fg$. We select our $X,Y,Z$ basis as follows:
\begin{enumerate}
	\item Take $Z$ to be a unit vector if $\fz(\fg)$ (unique up to sign).
	\item Take $\tilde{Y}$ to be the unique unit vector in $[\fg,\fg]$ which satisfies $Q(\tilde{Y},Z)\geq 0$.
	\item Take $X$ to be the unique vector satisfying $[X,\tilde{Y}] = \tilde{Y},\ Q(X,\tilde{Y})=Q(X,Z)=0$.
	\item Let $k = \dfrac{|X|}{\sqrt{1-Q(\tilde{Y},Z)^2}}$, and take $Y=k\tilde{Y}$.
\end{enumerate}
This gives us
\begin{align*}
Q(X,X) &= c>0 & Q(X,Y) &= 0 & Q(X,Z) &= 0\\
Q(Y,Z) &= k\cdot Q(\tilde{Y},Z) = a\geq 0 & Q(Y,Y) &= c+a^2 & Q(Z,Z) &= 1
\end{align*}

We now realize this algebra as a subalgebra of $\gl_3(\bbR)$ as
\[
	X = \begin{pmatrix}
		0 & 0 & 0\\
		0 & 1 & 0\\
		0 & 0 & 0
	\end{pmatrix}
	\qquad
	Y = \begin{pmatrix}
		0 & 0 & 0\\
		0 & 0 & 1\\
		0 & 0 & 0
	\end{pmatrix}
	\qquad
	Z = \begin{pmatrix}
		1 & 0 & 0\\
		0 & 0 & 0\\
		0 & 0 & 0
	\end{pmatrix}
\]
The corresponding Lie group is
\[
G = \left\{
	\begin{pmatrix}
		e^z & 0 & 0\\
		0 & x & y\\
		0 & 0 & 1
	\end{pmatrix}
	: z,y\in\bbR,\ x>0
	\right\}
\]
Where in local coordinates $X=x\partial_x,\ Y = x\partial_y,\ Z=\partial_z$.

The resulting metric in these coordinates is
\[
	g = \begin{pmatrix}
		\frac{c}{x^2} & 0 & 0\\
		0 & \frac{c}{x^2} + \frac{a^2}{x^2} & \frac{a}{x}\\
		0 & \frac{a}{x} & 1
	\end{pmatrix}
\]
Note that if $a=0$, $G$ becomes the direct product of $\bbH^2$ and $\bbR$, so we will assume $a>0$.

We consider the projection $G\to\bbH^2$ given by quotienting by $\exp(t\cdot Z)$. Under this projection we have $f=0$ $g = cg_{Hyp}$ and $F = a\omega_{Hyp}$. For convenience, we take $g_0 = ag_{Hyp}$, which gives us $\lambda = -1/a$ and $u(0) = c/a$.

The resulting implicit solution to the Ricci flow is
\[
	e^{4f}+\frac{ae^{2f}}{u} = 1+\frac{a^2}{c}
\]
\end{ex}

One question we are interested in is determining which of the other solvable 3-dimensional Lie groups can be realized as a bundle in the style of Theorem~\ref{thm:KE}.

\bibliographystyle{amsalpha}
\bibliography{References}
\end{document}